\input ./style/arxiv-general.cfg
\documentclass[seceqn,MSNbibl,number,citesort,dvips]{arxbj}
\makeatletter
   \@ifpackageloaded{graphicx}{}{\usepackage{graphicx}}
\makeatother
\usepackage{upgreek}

\aid{0}
\volume{22}
\issue{1}
\pubyear{2016}
\firstpage{38}
\lastpage{59}
\doi{10.3150/14-BEJ620} 

\makeatletter

\newcommand{\rrvert}{\vert}
\newcommand{\rrVert}{\Vert}
\newcommand{\llvert}{\vert}
\newcommand{\llVert}{\Vert}

\newtheorem{theorem}{Theorem}[section]
\newtheorem{conjecture}[theorem]{Conjecture}
\newtheorem{lemma}[theorem]{Lemma}
\newtheorem{proposition}[theorem]{Proposition}
\newtheorem{corollary}[theorem]{Corollary}

\newremark{example}[theorem]{Example}
\newremark{remark}[theorem]{Remark}

\newcommand{\E}{\mathbb{E}}

\newcommand{\Varhat}{\widehat{\operatorname{var}}}

\newcommand{\ND}{\mathcal{N}}
\newcommand{\eqdist}{\stackrel{d}{=}}
\newcommand{\convdist}{\stackrel{d}{\longrightarrow}}

\newcommand{\eqref}[1]{(\ref{#1})}

\newcommand{\rank}{\operatorname{rank}}
\newcommand{\Beta}{\operatorname{Beta}}
\newcommand{\trace}{\operatorname{tr}}
\newcommand{\Uniform}{\operatorname{Uniform}}
\newcommand{\degree}{\operatorname{degree}}
\newcommand{\diag}{\operatorname{diag}}
\newcommand{\Res}{\operatorname{Res}}
\newcommand{\Var}{\operatorname{var}}
\newcommand{\Cor}{\operatorname{cor}}
\renewcommand{\epsilon}{\varepsilon}

\newcommand{\fracd}[2]{({#1}/{#2})}
\makeatother

\begin{document}
\begin{frontmatter}

\title{Wald tests of singular hypotheses}
\runtitle{Wald tests of singular hypotheses}

\begin{aug}
\author[A]{\inits{M.}\fnms{Mathias}~\snm{Drton}\corref{}\thanksref{A}\ead[label=e1]{md5@uw.edu}} \and
\author[B]{\inits{H.}\fnms{Han}~\snm{Xiao}\thanksref{B}\ead[label=e2]{hxiao@stat.rutgers.edu}}
\address[A]{Department of Statistics, University of Washington,
Seattle, WA 98195-4322, USA.\\ \printead{e1}}
\address[B]{Department of Statistics \& Biostatistics, Rutgers
University, Piscataway, NJ 08854, USA.\\ \printead{e2}}
\end{aug}

\received{\smonth{6} \syear{2013}}
\revised{\smonth{2} \syear{2014}}

%
\begin{abstract}
Motivated by the problem of testing tetrad constraints in factor
analysis, we study the large-sample distribution of Wald statistics
at parameter points at which the gradient of the tested constraint
vanishes. When based on an asymptotically normal estimator, the
Wald statistic converges to a rational function of a normal random
vector. The rational function is determined by a homogeneous
polynomial and a covariance matrix.
For quadratic forms and bivariate monomials of arbitrary degree, we
show unexpected relationships to chi-square distributions that
explain conservative behavior of certain Wald tests. For general
monomials, we offer a conjecture according to which the reciprocal
of a certain quadratic form in the reciprocals of dependent normal
random variables is chi-square distributed.
\end{abstract}

%
\end{frontmatter}

\section{Introduction}
\label{sec:intro}

Let $f\in\mathbb{R}[x_1,\dots,x_k]$ be a homogeneous $k$-variate
polynomial with gradient $\nabla f$, and let $\Sigma$ be a $k\times k$
positive semidefinite matrix with positive diagonal entries. In this
paper, we study the distribution of the random variable
%
\begin{equation}
\label{eq:wald} W_{f,\Sigma} = \frac{f(X)^2}{(\nabla f(X))^T
\Sigma\nabla f(X)},
\end{equation}
where $X\sim\ND_k(0,\Sigma)$ is a normal random vector with zero mean
and covariance matrix $\Sigma$. The random variable $W_{f,\Sigma}$
arises in the description of the large-sample behavior of Wald tests
with $\Sigma$ being the asymptotic covariance matrix of an estimator
and the polynomial $f$ appearing in a Taylor approximation to the
function that defines the constraint to be tested.

In regular settings, the Wald statistic for a single constraint
converges to $\chi^2_1$, the chi-square distribution with one degree
of freedom. This familiar fact is recovered when $f(x)=a^Tx$,
$a\neq0$, is a linear form and
%
\begin{equation}
\label{eq:wald-chisquare} W_{f,\Sigma} = \biggl(\frac{a^TX}{\sqrt{a^T\Sigma a}} \biggr)^2
\end{equation}
becomes the square of a standard normal random variable; the vector
$a$ corresponds to a nonzero gradient of the tested constraint. Our
attention is devoted to cases in which $f$ has degree two or larger.
These singular cases occur when the gradient of the constraint is zero
at the true parameter.

For likelihood ratio tests, a large body of literature starting with
Chernoff \cite{chernoff1954} describes large-sample behavior in irregular
settings; examples of recent work are Aza{\"{\i}}s, Gassiat and   Mercadier \cite{azais2006},
Drton \cite{drton2009}, Kato and Kuriki \cite{kato2013} and Ritz and Skovgaard \cite{ritz2005}. In
contrast, much less work appears to exist for Wald tests. Three
examples we are aware of are Glonek \cite{glonek1993}, Gaffke, Steyer and von Davier \cite{gaffke1999}
and Gaffke, Heiligers and   Offinger \cite{gaffke2002} who treat singular hypotheses that correspond
to collapsility of contingency tables and confounding in
regression.
A further example is the preprint of {Dufour}, {Renault} and   {Zinde-Walsh} \cite{dufour2013} that was
posted while this paper was under review. Motivated by applications
in times series analysis (e.g., Galbraith and   Zinde-Walsh \cite{galbraith1997}), these authors
address, in particular, the problem of divergence of the Wald statistic
for testing multiple singular constraints.
Our own interest in singular Wald tests is
motivated by the fact that graphical models with hidden variables are
singular (Drton, Sturmfels and  Sullivant \cite{drtonoberwolfach}, Chapter~4).

In graphical modeling, or more specifically in factor analysis, the
testing of so-called `tetrad constraints' is a problem of particular
practical relevance
(Bollen, Lennox and Dahly \cite{bollen2009}, Bollen and Ting \cite{bollen2000}, Hipp and Bollen \cite{hipp2003},
Silva \textit{et~al.} \cite{silva2006}, Spirtes, Glymour and   Scheines \cite{spirtes2000}).
This problem goes back to Spearman \cite{spearman1904}; for some of the
history see Harman \cite{harman1976}. The desire to better understand the
Wald statistic for a tetrad was the initial statistical motivation for
this work. We solve the tetrad problem in Section~\ref{sec:tetrad};
the relevant polynomial is quadratic, namely, $f(x)=x_1x_2-x_3x_4$.
However, many other hypotheses are of interest in graphical modeling
and beyond (Drton, Sturmfels and  Sullivant \cite{drton2007},
Drton, Massam and Olkin \cite{drton2008}, Sullivant, Talaska and  Draisma \cite{sullivant2010}, Zwiernik and Smith \cite{zwiernik2012}).
In principle, any homogeneous polynomial $f$ could arise in the
description of a large-sample limit and, thus, general distribution
theory for the random variable $W_{f,\Sigma}$ from (\ref{eq:wald})
would be desirable.

At first sight, it may seem as if not much concrete can be said about
$W_{f,\Sigma}$ when $f$ has degree two or larger. However, the
distribution of $W_{f,\Sigma}$ can in surprising ways be independent
of the covariance matrix $\Sigma$ even if $\degree(f)\ge2$.
Glonek \cite{glonek1993} was the first to shows this in his study of the
case $f(x)=x_1x_2$ that is relevant, in particular, for hypotheses
that are the union of two sets. Moreover, the asymptotic distribution
in this case is smaller than $\chi^2_1$, making the Wald test maintain
(at times quite conservatively) a desired asymptotic level across the
entire null hypothesis. We will show that similar phenomena hold also
in degree higher than two; see Section~\ref{sec:monomial} that treats
monomials $f(x)=x_1^{\alpha_1}x_2^{\alpha_2}$. For the tetrad,
conservativeness has been remarked upon in work such
as Johnson and Bodner \cite{johnson2007}. According to our work in
Section~\ref{sec:tetrad}, this is due to the singular nature of the
hypothesis rather than effects of too small a sample size. We remark
that in singular settings standard $n$-out-of-$n$ bootstrap tests may
fail to achieve a desired asymptotic size, making it necessary to
consider $m$-out-of-$n$ and subsampling procedures; compare the
discussion and references in Drton and Williams \cite{drton2011}.

In the remainder of this paper, we first clarify the connection between
Wald tests and the random variables $W_{f,\Sigma}$
from (\ref{eq:wald}); see Section~\ref{sec:wald-tests}. We then study
the case of degree two. General aspects of quadratic forms $f$ are
the topic of Section~\ref{sec:quad}, which includes a full
classification of the bivariate case. The tetrad is studied in
Section~\ref{sec:tetrad}. As a first case of higher degrees, we treat
bivariate monomials $f$ of arbitrary degree in
Section~\ref{sec:monomial}. Our proofs make heavy use of the polar
coordinate representation of a pair of independent standard normal
random variables and, unfortunately, we have so far not been able to
prove the following conjecture, which we discuss further in
Section~\ref{sec:conjecture}, before giving final remarks in
Section~\ref{sec:conclusion}.

\begin{conjecture}
\label{conj:monomial}
Let $\Sigma$ be any positive semidefinite $k\times k$ matrix with
positive diagonal entries. If
$f(x_1,\ldots,x_k)=x_1^{\alpha_1}x_2^{\alpha_2}\cdots
x_k^{\alpha_k}$ with nonnegative integer exponents
$\alpha_1,\ldots,\alpha_k$ that are not all zero, then
\[
W_{f,\Sigma} \sim\frac{1}{(\alpha_1+\cdots+\alpha_k)^{2}} \chi^2_1.
\]
\end{conjecture}

It is not difficult to show that the conjecture holds when $\Sigma$ is
diagonal, \textit{that is}, the components of $X$ are independent Gaussian
random variables.

\begin{pf*}{Proof of Conjecture  \ref{conj:monomial} under independence}
Let $Z$ be a standard normal random variable, and $\alpha>0$. Then
$\alpha^2/Z^2$ follows the one-sided stable distribution of index
$\frac{1}{2}$ with parameter $\alpha$, which has the density
%
\begin{equation}
\label{eq:stable} p_{\alpha}(x) = \frac{\alpha}{\sqrt{2\uppi}}\frac{1}{\sqrt{x^3}}
\mathrm{e}^{-\fracd
{1}{2}\alpha^2/x},\qquad x>0.
\end{equation}
The law in (\ref{eq:stable}) is the distribution of the first
passage time of a Brownian motion to the level $\alpha$
(Feller \cite{feller1966}). Hence, it has the convolution rule
%
\begin{equation}
\label{eq:stable_conv} p_{\alpha}\ast p_{\beta} = p_{\alpha+\beta},\qquad \alpha,
\beta>0.
\end{equation}
When $f(x)=x_1^{\alpha_1}\cdots x_k^{\alpha_k}$ and
$\Sigma=(\sigma_{ij})$ is diagonal with
$\sigma_{11},\dots,\sigma_{kk}>0$, then
%
\begin{equation}
\label{eq:mono_indep} \frac{1}{W_{f,\Sigma}} = \frac{\alpha_1^2\sigma_{11}}{X_1^2}+\cdots+\frac{\alpha_k^2\sigma
_{kk}}{X_k^2}.
\end{equation}
By (\ref{eq:stable_conv}), the distribution of $1/W_{f,\Sigma}$ is
$(\alpha_1+\cdots+\alpha_k)^2/Z^2$. Therefore,
%
\begin{equation}
\label{eq:chi2_ind} W_{f,\Sigma} \sim\frac{1}{(\alpha_1+\cdots+\alpha_k)^{2}} \chi^2_1,
\end{equation}
as claimed in Conjecture \ref{conj:monomial}.
\end{pf*}

The preceding argument can be traced back to Shepp \cite{shepp1964}; see
also Cohen \cite{cohen1981}, Reid \cite{reid1987}, Quine \cite{quine1994}, DasGupta and Shepp \cite{dasgupta2004}. However, if
$X$ is a dependent random
vector, the argument no longer applies.
The case $k=2$, $\alpha_1=\alpha_2=1$ and $\Sigma$ arbitrary was
proved by Glonek \cite{glonek1993}; see Theorem~\ref{thm:glonek} below. We
prove the general statement for $k=2$ as Theorem~\ref{thm:monomial}.

\begin{remark}
\label{rem:real-exponents}
When simplified as in (\ref{eq:mono_indep}), the random variable
$W_{f,\Sigma}$ is well-defined when allowing
$\alpha_1,\dots,\alpha_k$ to take nonnegative real as opposed to
nonnegative integer values, and the above proof under independence
goes through in that case as well. We may thus
consider the random variable $W_{f,\Sigma}$ for a `monomial'
$f(x_1,\ldots,x_k)=x_1^{\alpha_1}x_2^{\alpha_2}\cdots
x_k^{\alpha_k}$ with $\alpha_1,\ldots,\alpha_k$ nonnegative real.
To be precise, we then refer to $W_{f,\Sigma}$ rewritten as
%
\begin{equation}
\label{eq:reciproque} W_{f,\Sigma} = \Biggl(\sum_{i=1}^k
\sum_{j=1}^k \sigma_{ij}
\frac{\alpha_i\alpha_j}{X_iX_j} \Biggr)^{-1}.
\end{equation}
With this convention, we believe Conjecture \ref{conj:monomial} to
be true for $\alpha_1,\ldots,\alpha_k$ nonnegative real.
\end{remark}

\section{Wald tests}
\label{sec:wald-tests}

To make the connection between Wald tests and the random variables
$W_{f,\Sigma}$ from (\ref{eq:wald}) explicit, suppose that
$\theta\in\mathbb{R}^k$ is a parameter of a statistical model and
that, based on a sample of size $n$, we wish to test the hypothesis
%
\begin{equation}
\label{eq:testing-problem} H_0\dvt\gamma(\theta)=0 \quad\mbox{versus}\quad H_1\dvt
\gamma(\theta)\neq0
\end{equation}
for a continuously differentiable function
$\gamma\dvtx \mathbb{R}^k\to\mathbb{R}$. Suppose further that there is a
$\sqrt{n}$-consistent estimator $\hat\theta$ of $\theta$ such that, as
$n\rightarrow\infty$, we have the convergence in distribution
\[
\sqrt{n}(\hat\theta-\theta) \convdist\ND_k\bigl(0,\Sigma(\theta)
\bigr),
\]
where the asymptotic covariance matrix $\Sigma(\theta)$ is a
continuous function of the parameter. The Wald statistic for
testing (\ref{eq:testing-problem}) is the ratio
%
\begin{equation}
\label{eq:wald-statistic} T_{\gamma}=\frac{\gamma(\hat\theta)^2}{\Varhat[\gamma(\hat
\theta)]} = \frac{n\gamma(\hat\theta)^2}{ (\nabla\gamma(\hat\theta))^T
\Sigma(\hat\theta) \nabla\gamma(\hat\theta)},
\end{equation}
where the denominator of the right-most term estimates the asymptotic
variance of $\gamma(\hat\theta)$, which by the delta method is given
by
\[
\bigl(\nabla\gamma(\theta)\bigr)^T \Sigma(\theta) \nabla\gamma(
\theta).
\]

Consider now a true distribution from $H_0$, that is, the true
parameter satisfies $\gamma(\theta)=0$. Without loss of generality,
we assume that $\theta=0$. If the gradient is nonzero at the true
parameter, then the limiting distribution of $T_{\gamma}$ is the
distribution of the random variable in (\ref{eq:wald-chisquare}) with
$a=\nabla\gamma(0)\neq0$ and $\Sigma=\Sigma(0)$. Hence, the limit
is $\chi^2_1$. However, if $\nabla\gamma(0)=0$ (i.e., the constraint
$\gamma$ is singular at the true parameter), then the asymptotic
distribution of $T_{\gamma}$ is no longer $\chi^2_1$ but rather given
by (\ref{eq:wald}) with the polynomial $f$ having higher degree; the
degree of $f$ is determined by how many derivatives of $\gamma$ vanish
at the true parameter.

\begin{proposition}
\label{prop:wald_normal}
Assume that $\gamma(0)=0$ and that there is a homogeneous polynomial
$f$ of degree $d\geq2$ such that, as $x\rightarrow0$,
\[
\gamma(x)=f(x)+\mathrm{o}\bigl(\llVert x\rrVert ^{d/2}\bigr)\quad
\mbox{and} \quad\nabla\gamma(x) = \nabla f(x) + \mathrm{o}\bigl(\llVert x\rrVert
^{(d-1)/2}\bigr).
\]
If $\sqrt{n}\hat\theta\convdist\mathcal{N}(0,\Sigma)$,
then $T_{\gamma}\convdist W_{f,\Sigma}$.
\end{proposition}

\begin{example}
\label{ex:glonek}
Glonek \cite{glonek1993} studied testing collapsibility properties of
$2\times2\times2$ contingency tables. Under an assumption of no
three-way interaction, collapsibility with respect to a chosen
margin amounts to the vanishing of at least one of two pairwise
interactions, which we here simply denote by $\theta_1$ and
$\theta_2$. In the $(\theta_1,\theta_2)$-plane, the hypothesis is
the union of the two coordinate axes, which can be described as the
solution set of $\gamma(\theta_1,\theta_2)=\theta_1\theta_2=0$ and
tested using the Wald statistic $T_\gamma$ based on maximum
likelihood estimates of $\theta_1$ and $\theta_2$. The hypothesis
is singular at the origin as reflected by the vanishing of $\nabla
\gamma$ when $\theta_1=\theta_2=0$. Away from the origin,
$T_\gamma$ has the expected asymptotic $\chi^2_1$ distribution. At
the origin, by Proposition~\ref{prop:wald_normal}, $T_\gamma$
converges to $W_{f,\Sigma}$, where $f(x)=x_1x_2$ and $\Sigma$ is the
asymptotic covariance matrix of the two maximum likelihood
estimates. The main result of Glonek \cite{glonek1993}, stated as a
theorem below, gives the distribution of $W_{f,\Sigma}$ in this
case. Glonek's surprising result clarifies that the Wald test for
this hypothesis is conservative at (and in finite samples near) the
intersection of the two sets making up the null hypothesis.
\end{example}

\begin{theorem}[(Glonek \cite{glonek1993})]\label{thm:glonek}
If $f(x)=x_1x_2$ and $\Sigma$ is any positive semidefinite $2\times
2$ matrix with positive diagonal entries, then
\[
W_{f,\Sigma} \sim\tfrac{1}{4} \chi^2_1.
\]
\end{theorem}

Before turning to concrete problems, we make two simple observations
that we will use to bring $(f,\Sigma)$ in convenient form.

\begin{lemma}
\label{lem:invariance}
Let $f\in\mathbb{R}[x_1,\dots,x_k]$ be a homogeneous polynomial, and
let $\Sigma$ be a positive semidefinite $k\times k$ matrix with
positive diagonal entries.
\begin{enumerate}[(ii)]
\item[(i)] If $c\in\mathbb{R}\setminus\{0\}$ is a nonzero scalar,
then $W_{cf,\Sigma} = W_{f,\Sigma}$.
\item[(ii)] If $B$ is an invertible $k\times k$ matrix, then
$W_{f\circ B,B^{-1}\Sigma B^{-T}}$ has the same distribution as
$W_{f,\Sigma}$.
\end{enumerate}
\end{lemma}

\begin{pf}
(i) Obvious, since $\nabla(cf) = c \nabla f$. (ii) Let
$X\sim\ND(0,\Sigma)$ and define
\[
Y= B^{-1}X\sim\ND\bigl(0,B^{-1}\Sigma B^{-T}
\bigr).
\]
Then $f(X)=(f\circ B)(Y)$ and $ \nabla(f\circ B)(Y)= B^T \nabla f
(X)$. Substituting into (\ref{eq:wald}) gives
\[
W_{f,\Sigma}=\frac{(f\circ B)(Y)^2}{(\nabla(f\circ B)(Y))^T
B^{-1}\Sigma B^{-T} \nabla(f\circ B)(Y)} = W_{f\circ
B,B^{-1}\Sigma B^{-T}}.
\]
\upqed\end{pf}

\section{Quadratic forms}
\label{sec:quad}

In this section, we consider the distribution of $W_{f,\Sigma}$ when
$f$ is a quadratic form, that is,
\[
f(x_1,x_2,\ldots,x_k) = \sum
_{1\le i\le j\le k} a_{ij}'x_ix_j
\]
for real coefficients $a_{ij}'$. Equivalently,
%
\begin{equation}
\label{eq:quad_form} f(x_1,x_2,\ldots,x_k) =
x^T Ax,
\end{equation}
where $A=(a_{ij})$ is symmetric, with $a_{ii}=a_{ii}'$ and
$a_{ij}=a_{ij}'/2$ for $ i<j$.

\subsection{Canonical form}
\label{sec:canonical-form}

Let $I$ denote the $k\times k$ identity matrix. We use the shorthand
\[
W_f := W_{f,I}
\]
when the covariance matrix $\Sigma$ is the identity.

\begin{lemma}
If $f\in\mathbb{R}[x_1,\dots,x_k]$ is homogeneous of degree $d$ and
$\Sigma$ is a positive semidefinite $k\times k$ matrix, then
$W_{f,\Sigma}$ has the same distribution as $W_{g}$ where $g$ is a
homogeneous degree $d$ polynomial in $\rank(\Sigma)$ many variables.
\end{lemma}

\begin{pf}
If $\Sigma$ has full rank, then $\Sigma=BB^T$ for an invertible
matrix $B$. Use Lemma~\ref{lem:invariance}(ii) to transform
$W_{f,\Sigma}$ to $W_{g}$ where $g=f\circ B$ is homogeneous of
degree $d$. If $\Sigma$ has\vspace*{2pt} rank $m<k$, then $\Sigma=B E_m B^T$,
where $B$ is invertible and $E_m$ is zero apart from the first $m$
diagonal entries that are equal to one. Form $g$ by substituting
$x_{m+1}=\cdots=x_k=0$ into $f\circ B$.
\end{pf}

Further simplifications are possible for a treatment of the random
variables $W_f$. In the case of quadratic forms, we may restrict
attention to canonical forms
$f(x)=\lambda_1x_1^2+\cdots+\lambda_kx_k^2$, as shown in the next
lemma.

\begin{lemma}
\label{lem:quadratic-forms-spectral}
Let $f(x)=x^TAx$ be a quadratic form given by a symmetric $k\times
k$ matrix $A\neq0$. If $\Sigma$ is a positive definite $k\times k$
matrix and $\lambda_1,\dots,\lambda_k$ are the eigenvalues of
$A\Sigma$, then $W_{f,\Sigma}$ has the same distribution as
%
\begin{equation}
\label{eq:canonical-quadratic} \frac{ ( \lambda_1 Z_1^2 + \cdots+ \lambda_k Z_k^2  )^2}{
4 (\lambda_1^2 Z_1^2+\cdots+\lambda_k^2 Z_k^2 )},
\end{equation}
where $Z_1,\dots,Z_k$ are independent standard normal random
variables.
\end{lemma}

\begin{pf}
Write $\Sigma=BB^T$ for an invertible matrix $B$. By
Lemma~\ref{lem:invariance}(ii), $W_{f,\Sigma}$ has the same
distribution as $W_g$ with $g(x)=x^T(B^TAB)x$. Let
\[
Q^T\bigl(B^TAB\bigr)Q =\diag(\lambda_1,
\dots,\lambda_k)
\]
be the spectral decomposition of $B^TAB$, with $Q$ orthogonal. Then
$\lambda_1,\dots,\lambda_k$ are also the eigenvalues of $A\Sigma$.
Applying Lemma~\ref{lem:invariance}(ii) again, we find that
$W_{f,\Sigma}$ has the same distribution as $W_h$ with
\[
h(x)=x^T\bigl(Q^TB^TABQ\bigr)x =
\lambda_1x_1^2+\cdots+\lambda_kx_k^2.
\]
Since $\nabla h(x) = 2(\lambda_1x_1,\dots, \lambda_kx_k)$, the claim
follows.
\end{pf}

In (\ref{eq:canonical-quadratic}), the set of eigenvalues
$\{\lambda_i\dvt 1\leq i\leq k\}$ can be scaled to $\{c\lambda_i\dvt 1\leq
i\leq k\}$ for any $c\neq0$, without changing the distribution;
recall also Lemma~\ref{lem:invariance}(i). For instance, we may scale
one nonzero eigenvalue to become equal to one. When all (scaled)
$\lambda_i$ are in $\{-1,1\}$, the description of the distribution of
$W_{f,\Sigma}$ can be simplified. We write $\Beta(\alpha,\beta)$ for
the Beta distribution with parameters $\alpha,\beta>0$.

\begin{lemma}
\label{thm:beta}
Let $k_1$ and $k_2$ be two positive integers, and let $k=k_1+k_2$.
If $f(x_1,\ldots,x_{k}) =
x_1^2+\cdots+x_{k_1}^2-x_{k_1+1}^2-\cdots-x_{k_1+k_2}^2$, then $W_f$ has
the same distribution as
\[
\tfrac{1}{4}R^2(2B-1)^2,
\]
where $R^2$ and $B$ are independent, $R^2\sim\chi^2_k$, and
$B\sim\Beta(k_1/2,k_2/2)$.
\end{lemma}

\begin{pf}
The distribution of $W_f$ is that of
\[
\frac{1}{4}\frac{(Z_1^2+\cdots+Z_{k_1}^2-Z_{k_1+1}^2-\cdots
-Z_{k_1+k_2}^2)^2}{
Z_1^2+\cdots+Z_k^2}
\]
with $Z_1,\ldots,Z_k$ independent and standard normal. Let
\[
Y_1:=Z_1^2+\cdots+Z_{k_1}^2
\sim\chi^2_{k_1},\qquad Y_2:=Z_{k_1+1}^2+
\cdots+Z_k^2 \sim\chi^2_{k_2}.
\]
Then $R^2:=Y_1+Y_2 \sim\chi^2_k$. Representing $Z_1,\ldots,Z_k$ in
polar coordinates shows that $R^2$ and $W_f/R^2$ are independent
(Muirhead \cite{muirhead1982}, Theorem~1.5.5). Since
$B=Y_1/(Y_1+Y_2)\sim\Beta(k_1/2,k_2/2)$, and
$(Y_1-Y_2)^2/(Y_1+Y_2)^2=(2B-1)^2$, we deduce that the two random
variables $W_f/R^2$ and $\frac{1}{4}(2B-1)^2$ have the same
distribution.
\end{pf}

We note that when $k=4$ and $k_1=k_2=2$, Lemma~\ref{thm:beta} gives
the equality of distributions
%
\begin{equation}
\label{eq:R2U2} W_f \eqdist\tfrac{1}{4}R^2U^2.
\end{equation}
The equality holds because, in this special case, $U=Y_1/(Y_1+Y_2)$ is
uniformly distributed on $[0,1]$, and $(2U-1)^2$ has the same
distribution as $U^2$. The distribution from (\ref{eq:R2U2}) will
appear in Section~\ref{sec:tetrad}.

For general eigenvalues $\lambda_i$, it seems that the distribution
from (\ref{eq:canonical-quadratic}) cannot be described in as simple
terms.

\subsection{Classification of bivariate quadratic forms}
\label{sec:quad_two}

We now turn to the bivariate case ($k=2$), that is, we are considering
a quadratic form in two variables,
\[
f(x_1,x_2)=ax_1^2+2bx_1x_2+cx_2^2.
\]
In this case, we are able to give a full classification of the
possible distributions of $W_f$ in terms of linear combinations of a
pair of independent $\chi^2_1$ random variables; see
Johnson, Kotz and   Balakrishnan \cite{johnson1994}, Section~18.8, for a discussion of such
distributions. Our classification reveals that for $k=2$ the
distributions for quadratic forms are stochastically bounded below and
above by $\chi^2_1/4$ and $\chi^2_2/4$, respectively.

\begin{theorem}
\label{thm:quadratic}
Let $\Sigma$ be a positive definite matrix, and let
$f(x_1,x_2)=ax_1^2+2bx_1x_2+cx_2^2$ be a nonzero quadratic form with
matrix
\[
A:= \pmatrix{ a & b
\cr
b&c } \neq0.
\]\vspace*{-12pt}
\begin{enumerate}[(b)]
\item[(a)] If $b^2-ac \geq0$, then
$W_{f,\Sigma}\sim\chi^2_1/4$.
\item[(b)] If $b^2-ac<0$, then
\[
W_{f,\Sigma} \eqdist \frac{1}{4} \biggl(Z_1^2+
\frac{4\det(A\Sigma)}{\trace(A\Sigma)^2} Z_2^2 \biggr),
\]
where $Z_1$ and $Z_2$ are independent standard normal random
variables.
\end{enumerate}
\end{theorem}

Before giving a proof of the theorem, we would like to point out that
the key insight, Lemma~\ref{lem:trig-positive-c} below, can also be
obtained from a theorem of Marianne Mora that is based on properties
of the Cauchy distribution (Seshadri \cite{seshadri1993}, Theorem~2.3).

\begin{pf*}{Proof of Theorem \ref{thm:quadratic}}
(a) When the discriminant $b^2-ac\geq0$ then $f$ factors into a
product of two linear forms. The joint distribution of the two
linear forms is bivariate normal. Write $\Sigma'$ for the
covariance matrix of the linear forms then the distribution of $W_f$
is equal to the distribution of $W_{g,\Sigma'}$ with
$g(x_1,x_2)=x_1x_2$. Hence, the distribution is $\chi^2_1/4$ by
Theorem~\ref{thm:glonek}. 

(b) In this case, the discriminant is negative and $f$ does not
factor. By Lemma~\ref{lem:quadratic-forms-spectral}, we can assume
$\Sigma=I$ and consider the distribution of $W_f$ for
$f(x_1,x_2)=\lambda_1x_1^2+\lambda_2x_2^2$, where $\lambda_1$ and
$\lambda_2$ are the eigenvalues of $A\Sigma$. Since
$\det(A\Sigma)=\lambda_1\lambda_2$ and
$\trace(A\Sigma)=\lambda_1+\lambda_2$, to prove the claim, we must
show that in this case
%
\begin{equation}
\label{eq:bivariate-neg-disc-claim} W_f \eqdist\frac{1}{4} \biggl(Z_1^2+
\frac{4\lambda_1\lambda
_2}{(\lambda
_1+\lambda_2)^2} Z_2^2 \biggr) = \frac{1}{4}
\biggl(Z_1^2+\frac{4c}{(1+c)^2} Z_2^2
\biggr),
\end{equation}
where $c=\lambda_2/\lambda_1>0$.


To show (\ref{eq:bivariate-neg-disc-claim}) we use the polar
coordinates again. So represent the two considered independent
standard normal random variables as $X_1=R\cos(\Psi)$ and
$X_2=R\sin(\Psi)$, where $R^2\sim\chi^2_2$ and
$\Psi\sim\Uniform[0,2\uppi]$ are independent. Then
\begin{eqnarray*}
W_f & =& \frac{R^2}{4} \cdot \frac{ [\cos(\Psi)^2+c\sin(\Psi)^2 ]^2}{
\cos(\Psi)^2+c^2\sin(\Psi)^2}
\\
& =& \frac{R^2}{4} \cdot \biggl(1-\frac{(1-c)^2}{(1+c)^2} \frac{(1+c)^2\cos(\Psi)^2\sin(\Psi)^2}{
\cos(\Psi)^2+c^2\sin(\Psi)^2}
\biggr).
\end{eqnarray*}
Using Lemma~\ref{lem:trig-positive-c}, we have
%
\begin{eqnarray}
W_f & \eqdist& \frac{R^2}{4} \biggl(1-\frac{(1-c)^2}{(1+c)^2}\cos(
\Psi)^2 \biggr)
\nonumber
\\
[-8pt]
\\
[-8pt] & =& \frac{R^2}{4} \biggl(\frac{4c}{(1+c)^2}\cos(
\Psi)^2+\sin(\Psi )^2 \biggr).
\nonumber
\end{eqnarray}
This is the claim from (\ref{eq:bivariate-neg-disc-claim}) because
$R\cos(\Psi)$ and $R\sin(\Psi)$ are independent and standard normal.
\end{pf*}

\begin{lemma}
\label{lem:trig-positive-c}
If $c\ge0$ and $\Psi$ has a uniform distribution over $[0,2\uppi]$,
then
\[
S_c(\Psi) := \frac{(1+c)^2\cos(\Psi)^2\sin(\Psi)^2}{\cos(\Psi)^2+c^2\sin
(\Psi)^2} \eqdist\cos(\Psi)^2.
\]
\end{lemma}

\begin{pf}
Let $R^2\sim\chi^2_2$ be independent of $\Psi$. Then $R\sin(\Psi)$
and $R\cos(\Psi)$ are independent and standard normal. Therefore,
\[
\frac{1}{R^2S_c(\Psi)} = \frac{1}{(1+c)^2} \cdot\frac{1}{[R\sin(\Psi)]^2} +
\frac{c^2}{(1+c)^2} \frac{1}{[R\cos(\Psi)]^2}
\]
is the sum of two independent random variables that follow the
one-sided stable distribution of index $\frac{1}{2}$. Since $c>0$,
the first summand has the stable distribution with parameter
$1/(1+c)$ and the second summand has parameter $c/(1+c)$. Hence, by
\eqref{eq:stable_conv}, their sum follows a stable law with
parameter 1. Expressing this in terms of the reciprocals,
\[
R^2S_c(\Psi) \eqdist R^2\cos(
\Psi)^2 \sim \chi^2_1.
\]
It follows that $S_c(\Psi)$ has the same distribution as
$\cos(\Psi)^2$. For instance, we may argue that $S_c(\Psi)$ and
$\cos(\Psi)^2$ have identical moments, which implies equality of the
distributions as both are compactly supported.
\end{pf}

The claim of Lemma~\ref{lem:trig-positive-c} is false for $c<0$.
Indeed, the distribution of $S_c(\Psi)$ varies with $c$ when $c<0$.

\subsection{Stochastic bounds}
\label{sec:stochastic-bounds}

To understand possible conservativeness of Wald tests, it is
interesting to look for stochastic bounds on $W_{f,\Sigma}$ that hold
for all $f$ and $\Sigma$. We denote the stochastic ordering of two
random variables as $U\le_{\mathrm{st}} V$ when $P(U> t)\le P(V> t)$ for all
$t\in\mathbb{R}$.

\begin{proposition}
\label{prop:csi}
If $f\in\mathbb{R}[x_1,\dots,x_k]$ is a quadratic form and $\Sigma$
any nonzero positive semidefinite $k\times k$ matrix, then
$W_{f,\Sigma}\le_{\mathrm{st}} \frac{1}{4}\chi^2_k$. Equality is achieved
when $f(x)=x_1^2+\cdots+x_k^2$ and $\Sigma$ is the identity matrix.
\end{proposition}

\begin{pf}
The second claim is obvious. For the first claim, without loss of
generality, we can restrict our attention to the distributions
from (\ref{eq:canonical-quadratic}). The Cauchy--Schwarz inequality
gives
\begin{eqnarray*}
\frac{ ( \lambda_1 Z_1^2 + \cdots+ \lambda_k Z_k^2  )^2}{
4 (\lambda_1^2 Z_1^2+\cdots+\lambda_k^2 Z_k^2 )} &\le& \frac{ ( Z_1^2 + \cdots+ Z_k^2  ) (\lambda_1^2
Z_1^2+\cdots+\lambda_k^2 Z_k^2 )}{
4 (\lambda_1^2 Z_1^2+\cdots+\lambda_k^2 Z_k^2 )}
\\
&= &\frac{1}{ 4} \bigl( Z_1^2 + \cdots+
Z_k^2 \bigr),
\end{eqnarray*}
which is the desired chi-square bound.
\end{pf}

The considered Wald test rejects the hypothesis that
$\gamma(\theta)=0$ when the statistic $T_\gamma$
from (\ref{eq:wald-statistic}) exceeds $c_\alpha$, where $c_\alpha$ is
the $(1-\alpha)$ quantile of the $\chi^2_1$ distribution. Let
$k_\alpha$ be the largest degrees of freedom $k$ such that a
$\frac{1}{4}\chi^2_k$ random variable exceeds $c_\alpha$ with
probability at most $\alpha$. According to
Proposition~\ref{prop:csi}, if the true parameter is a singularity at
which $\gamma$ can be approximated by a quadratic form in at most
$k_\alpha$ variables, then the Wald test is guaranteed to be
asymptotically conservative. Some values are
\[
k_{0.05} = 7,\qquad k_{0.025} = 11,\qquad k_{0.01} = 16,\qquad
k_{0.005}=20,\qquad k_{0.001} = 29.
\]

Turning to a lower bound, we can offer the following simple
observation.

\begin{proposition}
\label{lem:stochastic-lower-bound}
Suppose the quadratic form $f$ is given by a symmetric $k\times k$
matrix $A\neq0$, and suppose that $\Sigma$ is a positive definite
$k\times k$ matrix such that all eigenvalues of $A\Sigma$ are
nonnegative. Then $W_{f,\Sigma}\ge_{\mathrm{st}} \frac{1}{4}\chi^2_1$.
\end{proposition}

\begin{pf}
Let $\lambda_1,\dots,\lambda_k\ge0$ be the eigenvalues of
$A\Sigma$. By scaling, we can assume without loss of generality
that $\lambda_1=1$ and $0\le\lambda_i\leq1$ for $2\leq i\leq k$.
Then
\begin{eqnarray*}
\frac{ ( \lambda_1 Z_1^2 + \cdots+ \lambda_k Z_k^2  )^2}{
4 (\lambda_1^2 Z_1^2+\cdots+\lambda_k^2 Z_k^2 )} \ge \frac{Z_1^2 (\lambda_1^2 Z_1^2+\cdots+\lambda_k^2
Z_k^2 )}{
4 (\lambda_1^2 Z_1^2+\cdots+\lambda_k^2 Z_k^2 )} = \frac{1}{ 4}
Z_1^2,
\end{eqnarray*}
and the claim follows from Lemma~\ref{lem:quadratic-forms-spectral}.
\end{pf}

Proposition~\ref{lem:stochastic-lower-bound},
Theorem~\ref{thm:quadratic} and simulation experiments lead us to
conjecture that $\frac{1}{4}\chi^2_1$ is still a stochastic lower
bound when there are both positive and negative eigenvalues
$\lambda_i$.

\begin{conjecture}
For any quadratic form $f\neq0$ and any positive semidefinite
matrix $\Sigma\neq0$, the distribution of $W_{f,\Sigma}$
stochastically dominates $\frac{1}{4} \chi^2_1$.
\end{conjecture}

While we do not know how to prove this conjecture in general, we are
able to treat the special case where the eigenvalues $\lambda_i$ are
either 1 or $-1$.

\begin{theorem}
\label{thm:lbd}
Let $k_1,k_2>0$, and $k=k_1+k_2$. If $f(x_1,\ldots,x_{k}) =
x_1^2+\cdots+x_{k_1}^2-x_{k_1+1}^2-x_{k_1+k_2}^2$, then
$W_f\ge_{\mathrm{st}}\frac{1}{4}\chi^2_1$.
\end{theorem}

\begin{pf}
Without loss of generality, we assume $k_1\leq k_2$. If $k_1=0$ or
$k_1=k_2=1$, the claim follows
Proposition~\ref{lem:stochastic-lower-bound} and
Theorem~\ref{thm:quadratic}, respectively. We now consider the case
$k_1\geq1$ and $k_2\geq2$. By Lemma~\ref{thm:beta}, we know
%
\begin{equation}
\label{eq:lbd-piece1} W_f\eqdist\tfrac{1}{4}R^2(2B-1)^2,
\end{equation}
where $R^2$
and $B$ are independent, $R^2\sim\chi^2_k$, and
$B\sim\Beta(k_1/2,k_2/2)$. On the other hand, if
$B'\sim\Beta(1/2,(k-1)/2)$ and is independent of $R^2$, then
%
\begin{equation}
\label{eq:lbd-piece2} R^2B'\sim\chi^2_1.
\end{equation}
Let $g(x)$ and $h(x)$ be the density functions of $(2B-1)^2$ and
$B'$, respectively. The comparison of (\ref{eq:lbd-piece1})
and (\ref{eq:lbd-piece2}) shows that it suffices to prove that
$(2B-1)^2$ is stochastically larger than $B'$. We will show a
stronger result, namely, that the likelihood ratio $g(x)/h(x)$ is an
increasing function over $[0,1]$.

To simplify the argument, we rescale the density functions to
\begin{eqnarray*}
g(x)\sqrt{x}  &\propto& (1+\sqrt{x})^{k_1/2-1}(1-\sqrt{x})^{k_2/2-1}
\\
&&{} + (1-\sqrt{x})^{k_1/2-1}(1+\sqrt{x})^{k_2/2-1}
\end{eqnarray*}
and
\begin{eqnarray*}
h(x)\sqrt{x} & \propto&(1-x)^{(k-3)/2}
\\
& =& (1-\sqrt{x})^{(k_1+k_2-3)/2}(1+\sqrt{x})^{(k_1+k_2-3)/2}.
\end{eqnarray*}
For our purpose, it is equivalent to show the monotonicity of
$g(x^2)/h(x^2)$, which is proportional to
\begin{eqnarray*}
\ell(x) &:=& (1+x)^{(-k_2+1)/2}(1-x)^{(-k_1+1)/2}
\\
&&{} + (1-x)^{(-k_2+1)/2}(1+x)^{(-k_1+1)/2}.
\end{eqnarray*}
When $k_1=1$, the derivative of $\ell(x)$ satisfies
\[
2\ell'(x) = (k_2-1) (1-x)^{(-k_2-1)/2} -
(k_2-1) (1+x)^{(-k_2-1)/2}>0
\]
when $0<x<1$, and thus the likelihood ratio is an increasing
function. When $k_1\geq2$, we have
\begin{eqnarray*}
&&2\ell'(x)(1+x)^{(k_2+1)/2}(1-x)^{(k_2+1)/2}
\\
&&\quad= (1+x) \bigl[(k_2-1) (1+x)^{(k_2-k_1)/2} + (k_1-1)
(1-x)^{(k_2-k_1)/2} \bigr]
\\
&&\qquad {}- (1-x) \bigl[(k_1-1) (1+x)^{(k_2-k_1)/2} + (k_2-1)
(1-x)^{(k_2-k_1)/2} \bigr]
\\
&&\quad> (k_2-k_1) \bigl[(1+x)^{(k_2-k_1)/2} -
(1-x)^{(k_2-k_1)/2} \bigr]\geq0
\end{eqnarray*}
for all $0<x<1$. Therefore, $\ell(x)$ is an increasing function.
\end{pf}

\section{Tetrads}
\label{sec:tetrad}

We now turn to the problem that sparked our interest in Wald tests of
singular hypothesis, namely, the problem of testing tetrad constraints
on the covariance matrix $\Theta=(\theta_{ij})$ of a random vector $Y$
in $\mathbb{R}^p$ with $p\ge4$. A tetrad is a $2\times2$
subdeterminant that only involves off-diagonal entries and, without
loss of generality, we consider the tetrad
%
\begin{equation}
\label{eq:mytetrad} \gamma(\Theta)=\theta_{13}\theta_{24}-
\theta_{14}\theta_{23} = \det \pmatrix{ \theta_{13}
& \theta_{14}
\cr
\theta_{23} & \theta_{24} }.
\end{equation}

\begin{example}
Consider a factor analysis model in which the coordinates of $Y$ are
linear functions of a latent variable $X$ and noise terms. More
precisely, $Y_i=\beta_{0i}+\beta_i X +\epsilon_i$ where
$X\sim\ND(0,1)$ is independent of $\epsilon_1,\dots,\epsilon_p$,
which in turn are independent normal random variables. Then the
covariance between $Y_i$ and $Y_j$ is $\theta_{ij}=\beta_i\beta_j$
and the tetrad from (\ref{eq:mytetrad}) vanishes.
\end{example}

Suppose now that we observe a sample of independent and identically
distributed random vectors $Y^{(1)},\dots,Y^{(n)}$ with covariance matrix
$\Theta$. Let $\overline{Y}_n$ be the sample mean vector, and let
\[
\hat\Theta= \frac{1}{n}\sum_{i=1}^n
\bigl(Y^{(i)}-\overline Y_n\bigr) \bigl(Y^{(i)}-
\overline Y_n\bigr)^T
\]
be the empirical covariance matrix. Assuming that the data-generating
distribution has finite fourth moments,
it
holds that
\[
\sqrt{n}(\hat\Theta-\Theta) \convdist\ND_k\bigl(0,V(\Theta)\bigr)
\]
with $k=p^2$. The rows and columns of the asymptotic covariance
matrix $V(\Theta)$ are indexed by the pairs $ij:=(i,j)$, $1\le i,j\le
p$.
Since the tetrad from (\ref{eq:mytetrad}) only involves the
covariances indexed by the pairs in $C=\{13,14,23,24\}$, only the
principal submatrix
\[
\Sigma(\Theta) := V(\Theta)_{C\times C}
\]
is of relevance for the large-sample distribution of the sample tetrad
$\gamma(\hat\Theta)$.

The gradient of the tetrad is
\[
\nabla\gamma(\Theta) = (\theta_{24},-\theta_{23},-
\theta_{14},\theta_{13}).
\]
Hence, if at least one of the four covariances in the tetrad is
nonzero the Wald statistic $T_\gamma$ converges to a $\chi^2_1$
distribution. If, on the other hand,
$\theta_{13}=\theta_{14}=\theta_{23}=\theta_{24}=0$, then the
large-sample limit of $T_\gamma$ has the distribution of
$W_{f,\Sigma(\Theta)}$ where
\[
f(x)=x_1x_4-x_2x_3
\]
is a quadratic form in $k=4$ variables; recall
Proposition~\ref{prop:wald_normal}. This form can be written as
$x^TAx$ with a matrix that is a Kronecker product, namely
%
\begin{equation}
\label{eq:tetrad-A} A = \pmatrix{ 0 & 0 & 0&1
\cr
0 & 0 & -1&0
\cr
0 & -1&0 & 0
\cr
1&0
& 0&0 }= \pmatrix{ 0 & 1
\cr
-1 & 0 }\otimes \pmatrix{ 0 & 1
\cr
-1 & 0 }.
\end{equation}

If $Y$ is multivariate normal, then the asymptotic covariance
matrix has the entries
\[
V(\Theta)_{ij,kl} = \theta_{ik}\theta_{jl} +
\theta_{il}\theta_{jk}.
\]
In the singular case with
$\theta_{13}=\theta_{14}=\theta_{23}=\theta_{24}=0$, we have thus
\begin{eqnarray*}
\Sigma(\Theta) &=& \pmatrix{ \theta_{11}\theta_{33} &
\theta_{11}\theta_{34} & \theta_{12}
\theta_{33} & \theta_{12}\theta_{34}
\cr
\theta_{11}\theta_{34} & \theta_{11}
\theta_{44}& \theta_{12}\theta _{34} &
\theta_{12}\theta_{44}
\cr
\theta_{12}
\theta_{33} & \theta_{12}\theta_{34} &
\theta_{22}\theta_{33} & \theta_{22}
\theta_{34}
\cr
\theta_{12}\theta_{34} &
\theta_{12}\theta_{44}& \theta_{22}\theta
_{34} & \theta_{22}\theta_{44} }\\
&=& \pmatrix{
\theta_{11} & \theta_{12}
\cr
\theta_{12} &
\theta_{22} }\otimes \pmatrix{ \theta_{33} &
\theta_{34}
\cr
\theta_{34} & \theta_{44} },
\end{eqnarray*}
which again is a Kronecker product. We remark that $\Sigma(\Theta)$
would also be a Kronecker product if we had started with an elliptical
distribution instead of the normal, compare
Iwashita and Siotani \cite{iwashita1994}, equation (2.1), or if $(Y_1,Y_2)$ and $(Y_3,Y_4)$
were independent in the data-generating
distribution.

As we show next, in the singular case, the Kronecker structure of the
two matrices $A$ and $\Sigma(\Theta)$ gives a limiting
distribution of the Wald statistic for the tetrad that does not depend
on the block-diagonal covariance matrix $\Theta$.

\begin{theorem}
Let $\Sigma=\Sigma^{(1)}\otimes\Sigma^{(2)}$ be the Kronecker
product of two positive definite $2\times2$ matrices
$\Sigma^{(1)},\Sigma^{(2)}$. Let $f(x)=x_1x_4-x_2x_3$. Then
\[
W_{f,\Sigma} \eqdist\tfrac{1}{4}R^2U^2,
\]
where $R^2\sim\chi^2_4$ and $U\sim\Uniform[0,1]$ are
independent.
\end{theorem}

\begin{pf}
Since $f$ is a quadratic form we may consider the canonical form
from Lemma~\ref{lem:quadratic-forms-spectral}, which depends on the
(real) eigenvalues of $A\Sigma$. The claim follows from
Lemma~\ref{thm:beta} and the comments in the paragraph following its
proof provided the four eigenvalues of $A\Sigma$ all have the same
absolute value, two of them are positive and two are negative.

Let $\Sigma^{(i)}=(\sigma^{(i)}_{kl})$. Then, by
(\ref{eq:tetrad-A}),
\[
A\Sigma= \pmatrix{ -\sigma^{(1)}_{12} & \sigma^{(1)}_{11}
\vspace*{2pt}\cr
-\sigma^{(1)}_{22} & \sigma^{(1)}_{12} }
\otimes \pmatrix{ -\sigma^{(2)}_{12} & \sigma^{(2)}_{11}
\vspace*{2pt}\cr
-\sigma^{(2)}_{22} & \sigma^{(2)}_{12}
}.
\]
For $i=1,2$, since $\Sigma^{(i)}$ is positive definite, the matrix
\[
\pmatrix{ -\sigma^{(i)}_{12} & \sigma^{(i)}_{11}\vspace*{2pt}
\cr
-\sigma^{(i)}_{22} & \sigma^{(i)}_{12} }
\]
has the imaginary eigenvalues
\[
\pm\lambda^{(i)} =\pm\sqrt{\bigl(\sigma^{(i)}_{12}
\bigr)^2 - \sigma^{(i)}_{11}\sigma^{(i)}_{22}}
.
\]
It follows that $A\Sigma$ has the real eigenvalues
\[
\lambda^{(1)}\lambda^{(2)}\quad\mbox{and}\quad {-}\lambda^{(1)}
\lambda^{(2)},
\]
each with multiplicity two. Hence, Lemma~\ref{thm:beta}
applies with $k_1=k_2=2$.
\end{pf}

The distribution function of $\frac{1}{4} R^2U^2$ is
\[
F_{\mathrm{sing}}(t)= 1-\mathrm{e}^{-2t}+\sqrt{2\uppi t} \bigl(1-\Phi (2
\sqrt{t} ) \bigr),\qquad t\ge0,
\]
where $\Phi(t)$ is the distribution function of $\mathcal{N}(0,1)$.
The density $f_{\mathrm{sing}}(t)$ of $\frac{1}{4} R^2U^2$ is strictly
decreasing on $(0,\infty)$ and $f_{\mathrm{sing}}(t)\to\infty$ as $t\to0$.
In light of Theorem~\ref{thm:quadratic}, it is interesting to note
that the distribution of $\frac{1}{4} R^2U^2$ is not the distribution
of a linear combination of four independent $\chi^2_1$ random
variables, because the $\chi^2_d$ distribution has a finite density at
zero when $d\geq2$. However, the distribution satisfies
\[
\tfrac{1}{4}\chi^2_1 \le_{\mathrm{st}}
\tfrac{1}{4} R^2U^2 \le_{\mathrm{st}}
\tfrac{1}{4}\chi^2_2.
\]
The first inequality holds according to Theorem~\ref{thm:lbd}. The
second inequality holds because $R^2U\sim\chi^2_2$. According to the
next result, the distribution is also no larger than a $\chi^2_1$
distribution, which means that the Wald test of a tetrad constraint is
asymptotically conservative at the tetrad's singularities (which are
given by block-diagonal covariance matrices).

\begin{proposition}
\label{prop:tetrad-bound}
Suppose $R^2\sim\chi^2_4$ and $U\sim\Uniform[0,1]$ are independent.
Then
\[
\tfrac{1}{4}R^2U^2 \le_{\mathrm{st}}
\chi^2_1.
\]
\end{proposition}

\begin{pf}
Let $Z_1,\dots,Z_4$ be independent standard normal random variables.
Then the sum of squares
\[
Z_1^2+Z_2^2+Z_3^2+Z_4^2
\eqdist R^2 \sim\chi^2_4
\]
and the ratio
\[
\frac{Z_1^2}{Z_1^2+Z_2^2+Z_3^2+Z_4^2} \sim\Beta \biggl(\frac{1}{2},\frac{3}{2} \biggr)
\]
are independent.
Hence, the claim holds if and only if
\[
\tfrac{1}{2} U \le_{\mathrm{st}} \sqrt{B},
\]
where $U\sim\Uniform[0,1]$ and
$B\sim\Beta (\frac{1}{2},\frac{3}{2} )$. The distribution
of $U/2$ is supported on the interval $[0,1/2]$ on which it has
distribution function
\[
F_{U/2}(t)=2t.
\]
For $t\in(0,1)$, the distribution function of $\sqrt{B}$ has first
and second derivative
\[
F_{\sqrt{B}}'(t) = \frac{4 \sqrt{1 - t^2}}{\uppi}\quad\mbox{and}\quad
F_{\sqrt{B}}''(t) =-\frac{4t}{\uppi\sqrt{1 - t^2}}.
\]
Hence, $F_{\sqrt{B}}$ is strictly concave on $(0,1)$ and has a
tangent with slope $4/\uppi<2$ at $t=0$. Consequently,
\[
F_{U/2}(t)\ge F_{\sqrt{B}}(t),\qquad t\in\mathbb{R},
\]
giving the claimed ordering of $\frac{1}{4}R^2U^2$ and the
$\chi^2_1$ distribution.
\end{pf}


\section{Bivariate monomials}\label{sec:monomial}

In this section, we study the random variable $W_{f,\Sigma}$ when
$f(x)=x_1^{\alpha_1}x_2^{\alpha_2}$. If the exponents
$\alpha_1,\alpha_2$ are positive integers, then $f$ is a bivariate
monomial. However, all our arguments go through for a slightly more
general case in which $\alpha_1,\alpha_2$ are positive real numbers;
recall Remark~\ref{rem:real-exponents}.
Our main result is that the distribution of $W_{f,\Sigma}$ does not
depend on $\Sigma$.

\begin{theorem}
\label{thm:monomial}
Let $f(x)=x_1^{\alpha_1}x_2^{\alpha_2}$ with $\alpha_1,\alpha_2>0$,
and let $\Sigma$ be any positive semidefinite $2\times2$ matrix
with positive diagonal entries. Then
\[
W_{f,\Sigma} \sim\frac{1}{(\alpha_1+\alpha_2)^2} \chi^2_1.
\]
\end{theorem}

\begin{pf}
As shown in Section~\ref{sec:intro}, the claim is true if
$\Sigma=(\sigma_{ij})$ is diagonal. It thus suffices to show that
$W_{f,\Sigma}$ has the same distribution as $W_{f}:=W_{f,I}$.

By Lemma~\ref{lem:invariance}, we can assume without loss of
generality that $\sigma_{11}=\sigma_{22}=1$ and $\rho:=\sigma_{12}>
0$. Since
\[
\frac{1}{W_{f,\Sigma}} = \frac{\alpha_1^2}{X_1^2} + \frac{2\rho
\alpha
_1\alpha_2}{X_1X_2} +
\frac{\alpha_2^2}{X_2^2},
\]
we can also assume $\alpha_1=1$ for simplicity. With
$\sigma=1/\alpha_2$, we have
\[
\frac{1}{W_{f,\Sigma}} = \frac{1}{X_1^2} + \frac{2\rho}{\sigma
X_1X_2} + \frac{1}{\sigma^2 X_2^2}
\]
and need to show that
\[
W_{f,\Sigma} \sim\frac{1}{(1+\sigma^{-1})^{2}} \chi^2_1.
\]
If $\rho=1$, then $X_1$ and $X_2$ are almost surely equal and it is
clear that $W_{f,\Sigma}$ has the same distribution as $W_f$.
Hence, it remains to consider $0<\rho<1$.

Let $Z_1$ and $Z_2$ be independent standard normal random variables.
When expressing $Z_1=R\cos(\Psi)$ and $Z_2=R\sin(\Psi)$ in polar
coordinates, it holds that $R$ and $\Psi$ are independent, and
$\Psi$ is uniformly distributed over $[0,2\uppi]$. Let
$\rho=\sin(\phi)$ with $0 \le\phi< \uppi/2$, then the joint
distribution of $X_1$ and $X_2$ can be represented as
\[
X_1=R\cos(\Psi-\phi/2),\qquad X_2= R \sin(\Psi+\phi/2),
\]
which leads to
\[
\frac{1}{W_{f,\Sigma}} = \frac{1}{R^2} \cdot\frac{1}{T' }
\]
with
\[
\frac{1}{T'} =\frac{1}{\cos^2(\Psi-\phi/2)} + \frac{2\sin(\phi)}{\sigma\cos(\Psi-\phi/2) \sin(\Psi+\phi/2)} +
\frac{1}{\sigma^2 \sin^2(\Psi+\phi/2)}.
\]
Routine trigonometric calculations show that $T'$ can be expressed
as a function of the doubled angle $2\Psi$. More precisely,
\[
T' = \frac{\sigma^2}{4} t(2\Psi,\phi),
\]
where
\[
t(\psi,\phi) = \frac{
2-\cos(2\phi)+2\cos(\psi-\phi)-\cos(2\psi)-2\cos(\psi+\phi)
}{
1 - \sigma\cos(2 \phi) + (1 + \sigma) [\sigma+ \cos(\psi-\phi) -
\sigma\cos(\psi+\phi)]
}.
\]
Since $2\Psi$ is uniformly distributed on $[0,4\uppi]$, the
distribution of $T'$ is independent of $\phi$ if and only if the
same is true for the distribution of $T=t(\Psi,\phi)$.

We proceed by calculating the moments of $T$ and show that they
are independent of $\phi$. For each $0\le\phi<\uppi/2$, there exists
a small interval $L=[\phi-\epsilon,\phi+\epsilon]$ such that when
$m\geq1$, the function
\[
\sup_{\phi\in L} \bigl[t(\psi,\phi)\bigr]^{m-1}
\frac{\partial}{\partial
\phi} t(\psi,\phi)
\]
is integrable over $0\le\psi<2\uppi$. Therefore, we have
%
\begin{equation}
\label{eq:mom_deriv} \frac{\partial}{\partial\phi} \E\bigl(T^m\bigr) = \int
_{0}^{2\uppi} \frac{m}{2\uppi} \bigl[t(\psi,\phi)
\bigr]^{m-1}\frac{\partial}{\partial\phi} t(\psi,\phi)\,\mathrm{d}\psi.
\end{equation}
The expression of $\frac{\partial}{\partial\phi} t(\psi,\phi)$ is
long, so we omit it here.

We introduce the complex numbers $z=\mathrm{e}^{\mathrm{i}\psi}$ and $a=\mathrm{e}^{-\mathrm{i}\phi}$,
and express the functions $t(\psi,\phi)$ and
$\frac{\partial}{\partial\phi} t(\psi,\phi)$ in terms of $z$ and
$a$:
\begin{eqnarray*}
t(\psi,\phi) &=& u(z,a) =  \frac{(a-z)^2(1+az)^2}{z(a+a\sigma+a^2z-\sigma z)(-1+a^2\sigma
-az-a\sigma z)},
\\
\frac{\partial}{\partial\phi} t(\psi,\phi)&=& v(z,a) \\
&= & \frac{a(a-z)(1+az)(1+a^2s+2az-2a \sigma z+a^2z^2+\sigma
z^2)}{\mathrm{i}z(a+a\sigma+a^2z-\sigma z)^2(-1+a^2 \sigma-az-a \sigma
z)^2}
\\
&&{} \times\bigl(-a-a\sigma-z-a^2z+\sigma z+a^2 \sigma z
-az^2-a\sigma z^2\bigr).
\end{eqnarray*}
The integral in (\ref{eq:mom_deriv}) can be computed as a complex
contour integral on the unit circle $\mathbb{T}=\{z\dvt  \llvert z\rrvert =1\}$
\[
\int_{0}^{2\uppi} \bigl[t(\psi,\phi)
\bigr]^{m-1} \frac{\partial}{\partial
\phi} t(\psi,\phi)\,\mathrm{d}\psi =
\oint_{\mathbb{T}} \bigl[u(z,a)\bigr]^{m-1}v(z,a) \frac{1}{\mathrm{i}z}\,
\mathrm{d}z.
\]
Let
\[
q(z,a)= \bigl[u(z,a)\bigr]^{m-1}v(z,a)\frac{1}{\mathrm{i}z} .
\]
As a function of $z$, it has three poles of the same order $m+1$:
\[
z_0=0,\qquad z_1=\frac{a^2\sigma-1}{a+a\sigma},\qquad z_2 =
\frac{a+a\sigma
}{\sigma-a^2}.
\]
Since $a=\mathrm{e}^{-\mathrm{i}\phi}$ with $0\leq\phi<\uppi/2$ and $\sigma>0$, we have
\[
\llvert z_1\rrvert =\biggl\llvert \frac{a^2\sigma-a\bar a}{a+a\sigma}\biggr\rrvert =
\biggl\llvert \frac{a\sigma-\bar a}{1+\sigma}\biggr\rrvert <1,
\]
and similarly $\llvert z_2\rrvert >1$. Therefore, $q(z,a)$ has two poles $z_0$ and
$z_1$ within the unit disc. By the Residue theorem, we know
%
\begin{equation}
\label{eq:int_res} \frac{1}{2\uppi \mathrm{i}}\oint_{\mathbb{T}} q(z,a)\,\mathrm{d}z =
\Res(q;0) + \Res(q;z_1),
\end{equation}
where $\Res(q;0)$ and $\Res(q;z_1)$ are the residues at 0 and $z_1$
respectively. Let $\zeta_0=\{c\mathrm{e}^{\mathrm{i}\psi}, 0\leq\psi\leq2\uppi\}$ be
a small circle around 0 such that $z_1$ is outside the circle. Let
$S$ be the M\"obius transform
\[
S(w)=\frac{z_1-w}{1-\bar{z}_1w}.
\]
Then $S$ is one-to-one from the unit disk onto itself and maps $0$
to $z_1$, and $\zeta_0$ to a closed curve
$\zeta_1=\{S(c\mathrm{e}^{\mathrm{i}\psi}), 0\leq\psi\leq2\uppi\}$ around $z_1$ with
winding number one. It holds that
\[
\Res(q;z_1) = \frac{1}{2\uppi \mathrm{i}}\oint_{\zeta_1} q(z,a)\,
\mathrm{d}z = \frac
{1}{2\uppi
\mathrm{i}}\oint_{\zeta_0} q\bigl(S(w),a
\bigr)S' (w)\,\mathrm{d}w.
\]
It also holds that
\[
q\bigl(S(w),a\bigr)S' (w) = -q(w,a).
\]
The preceding identity is crucial for the proof, as it implies that
\[
\frac{1}{2\uppi \mathrm{i}}\oint_{\zeta_0} q\bigl(S(w),a\bigr)S' (w)\,
\mathrm{d}w = -\frac{1}{2\uppi
\mathrm{i}}\oint_{\zeta_0} q(w,a)\,\mathrm{d}w = -
\Res(q;0).
\]
Hence, the integral in (\ref{eq:int_res}) is zero.

We have shown that the integral in (\ref{eq:mom_deriv}) is zero for
every $m\geq1$, which means that the moments of $T$ do not depend
on $\phi$ for $0\leq\phi<\uppi/2$. When $\phi=0$, the random
variable $T$ is bounded, so its moments uniquely determine the
distribution. Therefore, the distribution of $T$ does not depend on
$\phi$, and the proof is complete.
\end{pf}

\begin{remark}
If $\alpha_1=\alpha_2$, then Theorem~\ref{thm:monomial} reduces to
Theorem~\ref{thm:glonek}. In this case, our proof above would only
need to treat $\sigma=1$. Glonek's proof of
Theorem~\ref{thm:glonek} finds the distribution function of a random
variable related to our $T$. If $\sigma=1$, this requires solving a
quadratic equation. When $\sigma\neq1$, we were unable to extend
this approach as a complicated quartic equation arises in the
computation of the distribution function. We thus turned to the
presented method of moments.
\end{remark}

Let $X=(X_1,X_2)^T $ and $Y=(Y_1,Y_2)^T $ be two independent
$\mathcal{N}_2(0,\Sigma)$ random vectors, where $\Sigma$ has positive
diagonal entries. Let $p_1,p_2$ be nonnegative numbers such that
$p_1+p_2=1$. The random variable
\[
Q=\frac{p_1X_2Y_1+p_2X_1Y_2}{\sqrt{(p_1X_2,p_2X_1)\Sigma
(p_1X_2,p_2X_1)^T }}
\]
has the standard normal distribution, and is independent of $X$. To
see this, observe that the conditional distribution of $Q$ given $X$
is always standard normal.
For $f(x)=x_1^{p_1}x_2^{p_2}$, let
\[
V_{f,\Sigma} = \frac{f(X)}{\sqrt{(\nabla f(X))^T \Sigma\nabla f(X)}}
\]
and $W_{f,\Sigma}=V_{f,\Sigma}^2$. Then
%
\begin{equation}
\label{eq:cauchy_2} p_1\frac{Y_1}{X_1} + p_2
\frac{Y_2}{X_2} = \frac{Q}{V_{f,\Sigma}}.
\end{equation}
By taking the conditional expectation given $V_{f,\Sigma}$, the
characteristic function of (\ref{eq:cauchy_2}) is seen to be
\[
\E \bigl[\exp\{\mathrm{i}tQ/V_{f,\Sigma}\} \bigr] = \E \bigl[ \exp\bigl\{-
\tfrac{1}{2}t^2/W_{f,\Sigma}\bigr\} \bigr].
\]
The uniqueness of the moment generating function for positive random
variables (Billingsley \cite{billingsley1995}, Theorem~22.2) yields that
(\ref{eq:cauchy_2}) has a standard Cauchy distribution (with
characteristic function $\mathrm{e}^{-\llvert t\rrvert }$) if and only if
$W_{f,\Sigma}\sim\chi^2_1$. Therefore, we have the following
equivalent version of Theorem~\ref{thm:monomial}.

\begin{corollary}
\label{thm:cauchy_2}
Let $X=(X_1,X_2)^T $ and $Y=(Y_1,Y_2)^T $ be independent
$\mathcal{N}_2(0,\Sigma)$ random vectors, where $\Sigma$ has
positive diagonal entries. If $p_1,p_2$ are nonnegative numbers such
that $p_1+p_2=1$, then the random variable
\[
p_1\frac{Y_1}{X_1} + p_2\frac{Y_2}{X_2}
\]
has the standard Cauchy distribution.
\end{corollary}

\section{Conjectures}
\label{sec:conjecture}

In Section~\ref{sec:monomial}, we mentioned that
Theorem~\ref{thm:monomial} and Corollary~\ref{thm:cauchy_2} are
equivalent. Similarly, Conjecture \ref{conj:monomial} is equivalent
to the following one.

\begin{conjecture}
Let $X=(X_1,X_2,\ldots,X_k)^T $ and $Y=(Y_1,Y_2,\ldots,Y_k)^T $ be
independent and have the same distribution
$\mathcal{N}_k(0,\Sigma)$, where $\Sigma$ has positive diagonal
entries. If $p_1,p_2,\ldots,p_k$ are nonnegative numbers such that
$p_1+p_2+\cdots+p_k=1$, then
\[
\label{eq:cauchy_k} \frac{p_1Y_1}{X_1}+\frac{p_2Y_2}{X_2}+\cdots+\frac{p_kY_k}{X_k}
\]
has the standard Cauchy distribution.
\end{conjecture}

For a proof of this conjecture, it is natural to try an induction type
argument, which might involve the ratio of normal random variables
with nonzero means (Marsaglia \cite{marsaglia1965}). However, we were unable to
make this work.


By taking the reciprocal of $W_{f,\Sigma}$, we can translate
Conjecture \ref{conj:monomial} into another
equivalent form.

\begin{conjecture}
\label{conj:reciprocal}
Let $X=(X_1,X_2,\ldots,X_k)^T \sim\mathcal{N}_k(0,\Sigma)$, where
$\Sigma$ has positive diagonal entries. If $p_1,p_2,\ldots,p_n$ are
nonnegative numbers such that $p_1+p_2+\cdots+p_n=1$, then
%
\begin{equation}
\label{eq:reciprocal} \biggl(\frac{p_1}{X_1},\frac{p_2}{X_2},\ldots,
\frac{p_n}{X_n} \biggr) \Sigma \biggl(\frac{p_1}{X_1},\frac{p_2}{X_2},
\ldots,\frac{p_n}{X_n} \biggr)^T 
\sim\frac{1}{\chi_1^2}.
\end{equation}
\end{conjecture}

Simulation provides strong evidence for the validity of these
conjectures. We have tried many randomly generated scenarios with
$2\le k\le5$, simulating large numbers of values for the rational
functions in question. In all cases, empirical distribution
functions were indistinguishable from the conjectured $\chi^2_1$ or
Cauchy distribution functions.


On the other hand, the positivity requirement for $p_1,p_2,\ldots,p_k$
is crucial for the validity of the conjectures. For
instance, let $Q$ be the reciprocal of the quantity on the left-hand
side of \eqref{eq:reciprocal}, and consider the special case where
$k=2$, $\Var(X_1)=\Var(X_2)=1$, $\Cor(X_1,X_2)=\rho$, and
$p_1=-p_2=1/2$. Assuming that $\llvert \rho\rrvert <1$, change coordinates to
\[
Z_1 = (X_1+X_2)/\sqrt{2(1+\rho)},\qquad
Z_2 = (X_1-X_2)/\sqrt{2(1-\rho)},
\]
and then to polar coordinates $Z_1=R\cos\Psi$ and
$Z_2=R\sin\Psi$. We obtain that
\[
Q=4 \biggl(\frac{1}{X_1^2}-\frac{2\rho}{X_1X_2}+\frac
{1}{X_2^2}
\biggr)^{-1} =R^2\frac{[\rho+\cos(2\Psi)]^2}{1-\rho^2}.
\]
The distribution of $Q$ now depends on $\rho$. For instance,
\[
\E[Q] = \frac{1 + 2 \rho^2}{1 - \rho^2}.
\]
%



\section{Conclusion}
\label{sec:conclusion}

In regular settings, the Wald statistic for testing a constraint on
the parameters of a statistical model converges to a $\chi^2_1$
distribution as the sample size increases. When the true parameter is
a singularity of the constraint, the limiting distribution is instead
determined by a rational function of jointly normal random variables
(recall Section~\ref{sec:wald-tests}). The distributions of these
rational functions are in surprising ways related to chi-square
distributions as we showed in our main results in
Sections~\ref{sec:quad}--\ref{sec:monomial}.

Our work led to several, in our opinion, intriguing conjectures about
the limiting distributions of Wald statistics. Although the
conjectures can be stated in elementary terms, we are not aware of any
other work that suggests these properties for the multivariate normal
distribution.

For quadratic forms, the usual canonical form leads to a particular
class of distributions parametrized by a collection of eigenvalues
(recall Lemma~\ref{lem:quadratic-forms-spectral}). It would be
interesting to study Schur convexity properties of this class of
distributions, which would provide further insights into asymptotic
conservativeness of Wald tests of singular hypotheses.

Finally, this paper has focused on testing a single constraint. A
natural follow-up problem is to study Wald tests of hypotheses defined
by several constraints. In this setting the choice of the constraints
representing a null hypothesis has an important effect on the
distribution theory, as exemplified by Gaffke, Steyer and von Davier \cite{gaffke1999} and
Gaffke, Heiligers and   Offinger \cite{gaffke2002}. As also mentioned in the \hyperref[sec:intro]{Introduction}, the new
work of {Dufour}, {Renault} and   {Zinde-Walsh} \cite{dufour2013} addresses some of the issues arising with
multiple constraints.

\section*{Acknowledgements}
We would like to thank G\'erard Letac and Lek-Heng Lim for helpful
comments on our conjectures. This work was supported by NSF under
Grant   DMS-0746265. Mathias Drton was also supported by an Alfred
P. Sloan Fellowship.



\printhistory
\end{document}